\documentclass[reqno,a4paper]{amsart}
\usepackage[utf8]{inputenc}
\usepackage[T1]{fontenc}
\usepackage[english]{babel}

\usepackage[allcolors=blue,colorlinks=true, pdfstartview=FitH, linkcolor=blue, pdftoolbar=true, bookmarks=true,bookmarksnumbered,plainpages]{hyperref}

\usepackage{graphicx}
\usepackage{caption}
\usepackage{subcaption}

\usepackage{indentfirst}
\usepackage{tablefootnote}
\usepackage[usenames,dvipsnames,svgnames,table]{xcolor}
\usepackage{amsmath,amssymb,amsthm,amsfonts}
\usepackage[slantedGreek]{mathptmx} %
\DeclareMathAlphabet{\mathcal}{OMS}{cmsy}{m}{n}
\usepackage{lipsum}

\usepackage{mathtools}
\usepackage[normalem]{ulem}

\usepackage{newtxtext}

\usepackage{enumitem}

\newtheoremstyle{theorem}
{6pt +1\p@ -2.0\p@}% above space (default)
{6pt +1\p@ -2.0\p@}% below space
{\it}			      % Body font
{}				  % Indent amount
{\bfseries}   % Theorem head font
{.}               % Punctuation after theorem head
{.4em}       % Space after theorem head
{}               % Theorem head spec (can be left empty, meaning `normal')

\newtheorem{definition}{Definition}[section]

\newtheorem{theorem}{Theorem}[section]

\newtheorem{lemma}{Lemma}[section]
\newtheorem{example}{Example}[section]
\newtheorem{corollary}{Corollary}[section]

\numberwithin{equation}{section}

\newcounter{remark}[section]
\newenvironment{remark}
{\refstepcounter{remark}\medskip\noindent{\sc Remark\ \thesection.\theremark:}}{\medskip}
\newcounter{alphatheo}

\renewenvironment{proof}{\medskip\noindent{\sc Proof:}}{\medskip}

\newcommand{\R}{\mathbb R}

\newcommand{\N}{\mathbb N}
\newcommand{\Z}{\mathbb Z}

\def\<{\langle}
\def\>{\rangle}

\newcommand {\be}{\begin{equation}}
	\newcommand {\ee}{\end{equation}}
\newcommand {\bee}{\begin{equation*}}
	\newcommand {\eee}{\end{equation*}}
\newcommand {\baa}{\begin{align*}}
	\newcommand {\eaa}{\end{align*}}

\begin{document}

\title[Stein-Weiss inequality in $L^{1}$ norm for vector fields]{Stein-Weiss inequality in $L^{1}$ norm for vector fields}

\author {Pablo De N\'apoli}
\address{IMAS (UBA-CONICET) and Departamento de Matem\'atica\\
	Facultad de Ciencias Exactas y Naturales - Universidad de Buenos Aires\\
	Ciudad Universitaria - 1428 Buenos Aires, Argentina\\
	pdenapo@dm.uba.ar}

\author {Tiago Picon}
\address{Departamento de  Computa\c{c}\~ao e Matem\'atica, Universidade de S\~ao Paulo, Ribeir\~ao Preto, S\~ao Paulo, 14040-901, Brasil}
\email{picon@ffclrp.usp.br}

\thanks{The first author was supported by ANPCyT under grant PICT-2018-03017, and by Universidad de
Buenos Aires under grant 20020160100002BA. He is a members of
CONICET, Argentina. The second author was supported by CNPq (grant 311430/2018-0) and FAPESP (grant 2018/15484-7).} 



\begin{abstract}
In this work, we investigate the limit case $p=1$ of the classical Stein--Weiss inequality for the Riesz potential.
We present a characterization for a special class of vector fields associated to cocanceling operators introduced by Van Schaftingen in \cite{VS}. As an application, we recover some div-curl inequalities found in \cite{BVS,HP, SSV,Z}. 
In addition, we discuss a  two-weight inequality with general weights in the scalar case, extending the previous result of Sawyer  \cite{Saw1, Saw2} to this case.
\end{abstract}

\maketitle

%\tableofcontents

%%%%%%%%%%%%%%%1111111111111111
\section{Introduction} 

In the Euclidean space $\R^N$, we consider for $0 < \ell < N$ the classical Riesz potential or fractional integral operator given by   
$$ I_\ell f(x) := \gamma_{N,\ell} \int_{\R^N} \frac{f(y)}{|x-y|^{N-\ell}} \; dy, \quad \quad f \in S(\R^N)$$
with $\displaystyle{\gamma_{N,\ell}=\pi^{\ell-\frac{N}{2}}\Gamma((N-\ell)/{2})/\Gamma({\ell}/{2})}$. 

Two-weight inequalities for this operator have been extensively studied in the literature, starting by the 
pioneering work due to E. Stein and G. Weiss for the case of power weights, that we recall bellow. 

\begin{theorem} (\cite{SW})
Let $N \geq 1$ and $1<p\leq q<\infty$. Assume $\alpha, \beta$ satisfying the conditions 
\begin{itemize}
\item[(i)]  $\displaystyle{\alpha<\frac{N}{p'}}$ and  $\displaystyle{\beta<\frac{N}{q}}$ with $\displaystyle{\frac{1}{p}+\frac{1}{p'}=1}$;
\item[(ii)] $\displaystyle{\alpha + \beta \ge 0}$;
\item[(ii)] $\displaystyle{\frac{1}{q}=\frac{1}{p}+\frac{\alpha+\beta-\ell}{N}}$. 
\end{itemize}
Then there exists $C>0$, depending only the parameters $p,q,\alpha,\beta$, such that 
\begin{equation}\label{ineq1.1}
\||x|^{-\beta} I_\ell f(x) \|_{L^q(\mathbb{R}^N)} \le C \||x|^\alpha f(x) \|_{L^p(\mathbb{R}^N)}, \quad \quad \forall f \in S(\R^N).
\end{equation}
\label{theorem-SW}
\end{theorem}
We remark that fixed $0<\ell<N$ and $1<p<\infty$, the restriction of the parameters in the previous theorem are sharp.  In fact, condition (i) is necessary to ensure integrability, (ii) follows from the scaling of the inequality (which is a special feature of the power-weights case), whereas the necessity of the assumption $\alpha+\beta \geq 0$ is consequence of a (non trivial) translation argument (see \cite{Saw3}). In the case of radially symmetric functions, this condition can be relaxed for 
$\alpha+\beta \geq (N-1) \left( \frac{1}{q}-\frac{1}{p}\right)$
(see \cite{Rubin, DDD1,D}).    

\medskip

Later on, more general two-weighted inequalities of the form
\begin{equation}\label{two-weight-ineq}
\left( \int_{\R^{N}} | I_\ell f(x)|^{q} \; u(x) \; dx\right)^{1/q} \le C \left( \int_{\R^{N}}|f(x)|^{p}\; v(x) \; dx \right)^{1/p}
\end{equation}
were considered in the literature. E.  Sawyer in  \cite{Saw1, Saw2} gave a characterization of \eqref{two-weight-ineq} in terms of the so-called \emph{testing conditions}. More precisely,  if $0<\ell<N$ and $1<p\leq q<\infty$ then Sawyer proved 
that \eqref{ineq1.1} holds if and only if for every ball $B$ 
\begin{equation}
\left( \int_{B} [I_{\ell}(\chi_{B}u)(x)]^{p'}v(x)dx \right)^{1/p'}\lesssim \footnote{The notation $f \lesssim g $ means that there exists a constant $C>0$ such that $f(x)\leq C g(x)$ for all $x \in \R^N$.}
\left( \int_{B} u(x)dx \right)^{1/q'}
\end{equation} 
and
\begin{equation}
\left( \int_{B} [I_{\ell}(\chi_{B}v)(x)]^{q}u(x)dx \right)^{1/q}\lesssim \left( \int_{B} v(x)dx \right)^{1/p}.
\end{equation} 
A different type of conditions, called \emph{bump-conditions} is also considered in the literature. 
For instance, we mention E. Sawyer and R. Wheeden proved in \cite[Theorem 1]{Saw3} that if $1<p\leq q<\infty$ and for some  $r>1$
\begin{equation}
 |B|^{\ell/N+1/q-1/p} \left( \frac{1}{|B|} \int_{B} [u(x)]^{r}dx \right)^{\frac{1}{rq}}
\left( \int_{B} [v(x)]^{(1-p^\prime)r} \; dx \right)^{\frac{1}{p^\prime r}} \leq C \; \hbox{for all balls}\; B \subset \R^N, 
\label{bump-condition}
\end{equation}
then \eqref{two-weight-ineq} holds. Whereas this condition is not necessary for \eqref{two-weight-ineq} to hold, it does not involve the operator $I_\ell$ itself, and it is  ``almost necessary'' in 
the sense that the validity of\eqref{two-weight-ineq} implies the analogue of \eqref{bump-condition} with $r=1$. The Stein--Weiss result 
(Theorem \ref{theorem-SW}) follows as a special case. Slightly more general bump conditions involving Orlicz norms were developed by C.
P\'erez in \cite{P}. We also refer the reader to \cite{CU} for a survey of the theory from the view-point of sparse majorization.

\medskip

We remark however that the Stein-Weiss inequality has some special features that do not hold for more general weights, 
notably the rotational and scaling invariance. These features for instance play a key role in proving the existence of extremal functions in some cases \cite{CT, DDS, Lieb}, or in proving improved versions for radial functions by the method of \cite{DDD1} or the more general inequalities with mixed norms in \cite{DL}. Other proofs of the 
Stein-Weiss inequality and generalizations to different contexts are presented in \cite{Beckner, CLLT, Dou, GGSS, KRS, Wang}. 
\medskip 

A natural question arises on the validity of the inequalities \eqref{ineq1.1} and \eqref{two-weight-ineq} for $p=1$. 
However the inequality \eqref{ineq1.1}  breaks down for $p=1$ and  $\alpha=0$ for instance. Indeed, let $\varphi \in C_{c}^{\infty}(\R^N)$ be a positive smooth function supported on the unit ball $B=B(0,1)$ such that $\int_{\R^{N}} \varphi(x)dx=1$ and for each $\epsilon > 0$ consider $\varphi_{\epsilon}(x):=\epsilon^{-N}\varphi(\epsilon^{-1}x)$. Then applying \eqref{ineq1.1} for $\varphi_{\epsilon}$ and using the scaling invariance, we have  
\begin{equation}\nonumber
\||x|^{-\beta} I_\ell \varphi_{\epsilon} \|_{L^q(\mathbb{R}^N)} \le C \| \varphi_{\epsilon}\|_{L^1(\mathbb{R}^N)} = C, %\quad \quad  f \in S(\R^N).
\end{equation}
uniformly. Taking $\epsilon \searrow 0$ we know that $I_{\ell}{\varphi_{\epsilon}}(x) \rightarrow |x|^{-N+\ell}$ almost everywhere that implies $\||x|^{-\beta+N-\ell}  \|_{L^q(\mathbb{R}^N)}\lesssim 1$, a  contradiction.  

\medskip 

Although Stein-Weiss inequality fails for $p=1$ and $\alpha=0$, some progress is still possible. D`Ancona and Luca in \cite[Theorem 1.3]{DL} (with special parameters $p=\tilde{p}=1$ and $q=\tilde{q}$), proved the validity of \eqref{ineq1.1} for $p=1$ when $\alpha+\beta>0$ and $\alpha<0$, or assuming some control of support of the Fourier transform $\widehat{f}$ on an annulus.    

\medskip

In this note we present new perspectives on the classical formulation of Stein-Weiss and Sawer`s inequality when $p=1$. Our first result is a characterization of two-weight in case $p=1$, that can be thought as a limit case 
of Sawyer's result. 
%

% We start by pointing out the weights $u(x):=|x|^{-\beta}$ and $v(x):=|x|^{\alpha}$ that compose the inequality \eqref{ineq1.1} develop a crucial role at the non validity of the inequality  when $p=1$.

\begin{theorem}\label{thm1.1} 
Let $0<\ell<N$, $1<q<\infty$ and $u(x),v(x)$ nonnegative functions in $L^{1}_{loc}(\R^N)$. The inequality
\begin{equation}\label{ineq1.2}
\left( \int_{\R^{N}} | I_\ell f(x)|^{q} \; u(x)dx\right)^{1/q} \le C \int_{\R^{N}}
|f(x)| \; v(x) \; dx
\end{equation}
holds for any $f\in C_{c}^{\infty}(\mathbb{R}^N)$ and some $C>0$ %$f\in L^1(\mathbb{R}^N, v(x)dx)$ and some $C>0$ 
independent of $f$  if and only if there exists $\widetilde{C}>0$ such that 
\begin{align}\label{hipotese}
\left( \int_{\R^N} \frac{u(x)}{|x-y|^{(N-\ell)q}}  \,dx \right)^{1/q} \leq \widetilde{C} \; v(y), \quad \quad \forall \;\; \text{for almost } \; y \in \R^N.
\end{align}
\end{theorem}

As a consequence, we recover the result due to D`Ancona and Luca in \cite{DL} for the case $p=1$, namely:

\begin{corollary} 
Let $\beta<N/q$,  $\alpha+\beta>0$ and $\displaystyle{{1}/{q}=1+ ({\alpha+\beta-\ell})/{N}}$. If $\alpha<0$
then the Stein-Weiss inequality \eqref{ineq1.1} holds.
\end{corollary}
Indeed, it is sufficient to verify that the weights $u(x):=|x|^{-\beta}$ and $v(x):=|x|^{\alpha}$ satisfy \eqref{hipotese}. In fact,
for $B_{R}:=B(0,R) $ we have that
\begin{align*}
\int_{\R^{N}}\frac{|x|^{-\beta q}}{|x-y|^{(N-\ell)q}}\,dx & \leq 
\int_{B_{2|y|}}\frac{|x|^{\alpha q}}{|x-y|^{(N-\ell)q}|x|^{N-(N-\ell)q}}\,dx
           + \int_{\R^{N} \backslash B_{2|y|}}{|x|^{\alpha q-N}}\,dx
& \lesssim |y|^{\alpha q},
\end{align*}
where the bound on the integral term on $B_{2|y|}$ follows from $N-(N-\ell)q=(\alpha+\beta)q>0$ and whereas the
condition $\alpha<0$ was used to bound the integral on $\R^N \backslash B_{2|y|}$. Thus, the desired conclusion follows.
\bigskip

Let us return to the classical inequality Stein-Weiss inequality with power weights \eqref{ineq1.1} which can be written as 
\be \left( \int_{\R^N} \left| \int_{\R^N} K(x,y)f(y) \; dy \right|^q |x|^{-\beta q} \;dx  \right)^{1/q} 
\leq C  \left( \int_{\R^N} |f(x)|^p |x|^{\alpha p} \; dx \right)^{1/p} \label{weighted1}\ee
where $K(x,y):=\gamma_{N,\ell} |x-y|^{-N+\ell}$ and $0<\ell<N$. Inequalities of this type for standard kernels $K(x,y)$ (see \eqref{ka} and \eqref{kb} below) were recently studied by Hounie and Picon in \cite{HP} for $N \geq 2$, $p=1$ and $\alpha=0$ in the setting of vector fields where $f(x)=A(D)u(x)\in C_{c}^{\infty}(\R^{N};F)$ for $u \in C_{c}^{\infty}(\R^{N};E)$ and $A(D)$ is an elliptic and canceling linear homogeneous differential operators of order $\nu$ with constant coefficients from a finite dimensional complex vector space $E$ to a finite dimensional complex vector space $F$. Here we recall that $A(D)$ is said to be elliptic if the symbol $A(\xi)$ is one-one for $\xi \in \R^{N} \backslash \left\{0 \right\}$ and that the canceling property means that
 \begin{equation}\label{canceling}
 \displaystyle{\bigcap_{\xi \in \R^{N}\backslash 0}\,A(\xi)[E]=\left\{ 0 \right\}}.
 \end{equation}
The rich theory of canceling operators is due to Van Schaftingen in \cite{VS}, where he characterizes the following Sobolev-Gagliardo-Nirenberg inequality
$$ \| D^{\nu-1} \; u \|_{L^{N/(N-1)}} \leq \| A(D) \; u \|_{L^1}, $$ 
for elliptic operators covering in particular several interesting a priori estimates for vector fields with divergence free and chain complexes (see also \cite{VS2}).   

One fundamental property of ellipitc and canceling operators $A(D)$ is the existence of a linear homogeneous differential operators $L(D): C^{\infty}(\R^N;F) \rightarrow C^{\infty}(\R^N;V)$ of order $m$ for some finite dimensional complex vector space $V$ such that %$L(D)\circ A(D)=0$ (see ***).  
$$\displaystyle{\bigcap_{\xi \in \R^{N}\backslash 0}\,ker \,L(\xi)\,=\,\bigcap_{\xi \in \R^{N}\backslash 0}\,A(\xi)[E]=\left\{ 0 \right\}}.$$ 
The following definition that was also introduced by Van Schaftingen in \cite{VS} is fundamental to our work:
\begin{definition}\label{de1.1}
 Let $L(D)$ be a homogeneous linear differential operators of order m on $\R^{N}$ from $F$ to $V$.
The operator $L(D)$ is cocanceling if 
$$\displaystyle{\bigcap_{\xi \in \R^{N}\backslash \left\{ 0 \right\}}ker\,L(\xi)=\left\{ 0 \right\}}.$$ 
\end{definition}
A simple example of cocanceling operator on $\R^N$ from $F=\R^N$ to $V=\R$ is the divergence operator $L(D)=div$. Indeed, for every $e \in \R^N$ we have $L(\xi)[e]=\xi \cdot e$ and then clearly 
$$\bigcap_{\xi \in \R^{N}\backslash \left\{ 0 \right\}} ker\,L(\xi)=\bigcap_{\xi \in \R^{N}\backslash \left\{ 0 \right\}} \xi^{\perp}=\left\{ 0 \right\}.$$

\medskip

Our second and main result in this paper is a characterization of the Stein-Weiss inequality when $p=1$ for vector fields associated to the kernel of cocanceling operators.

\begin{theorem}\label{main1}
Suppose $N\geq 2$, $0<\ell<N$, $0 \leq \alpha<1$, $\beta<N/q$, $\alpha+\beta>0$ and \\ $\displaystyle{\frac{1}{q}=1+ \frac{\alpha+\beta-\ell}{N}}$. Then if $L(D)$ is cocanceling  there exists $C>0$ such that 
\begin{equation}\label{main5}
 \left( \int_{\R^N} \left| I_{\ell}f(x) \right|^q |x|^{-\beta q} \;dx  \right)^{1/q} 
\leq C \int_{\R^{N}} |x|^{\alpha}|f(x)|dx, 
\end{equation}
for all $f \in C^\infty_c(\R^N;F)$ such that $L(D)f=0$ in the sense of distributions. Conversely, if for all $f \in C^\infty_c(\R^N;F)$ satisfying $L(D)f=0$ the inequality \eqref{main5} holds for $\alpha=0$ then $L(D)$ has to be cocanceling. 
\end{theorem}
We remark that an interesting open question that remains is whether or not the converse of \eqref{main5} holds for $0<\alpha<1$.

\bigskip

The paper is organized as follows.  Section \ref{two} is devoted to the proof of Theorem \ref{thm1.1}, in special we state some (possible well known) two-weight inequalities which could not find the proofs in the literature. The proof of  inequality \eqref{main5} of Theorem \ref{main1} is presented 
in Section \ref{sec3} that is consequence of a short improvement of \cite[Lemma 2.2 ]{BVS}  announced at Lema \ref{mainlemma} and the Fundamental Lema \ref{main}.  In Section \ref{sec4}, we show that the cocancelation property is necessary for $\alpha=0$ and in the Section \ref{sec5} we present a simple counter example for vector fields satisfying $div \,\vec{f}=0$ i.e. belonging to the kernel of a cocanceling operator such that \eqref{main5} fails to $\alpha=1$. Finally, in Section \ref{sec6}, we discuss some applications of Theorem \ref{main1} unifying and extending some results due to Hounie \& Picon \cite{HP}, Zhang \cite{Z} and Schikorra \& Spector \& Van Schaftnigen \cite{SSV}.

\section{Two-weight Stein-Weiss inequality in $L^1$ norm}\label{two}

This section is devoted to the proof of Theorem \ref{thm1.1}. The main ingredient is a weighted inequality that we state in a general setting.

\begin{lemma}\label{weight1}
Let $g \geq 0$ and $1 \leq q <\infty$. Then 
\begin{equation}\label{w1}
\left[ \int_{\R^{N}} \left(\int_{B(0,|x|/2)}g(z)dz \right)^{q}\widetilde{u}(x)dx \right]^{1/q} \leq C \int_{\R^N} g(x)\widetilde{v}(x)dx
\end{equation}
holds if and only if 
\begin{equation}\label{w2}
\tilde{A}:=\sup_{R>0}\left( \int_{B^{c}(0,R)}\widetilde{u}(x)dx \right)^{1/q}\left( \sup_{x \in B(0,R)}\widetilde{v}^{\,\,-1}(x) \right)<\infty.
\end{equation}
Analogously 
\begin{equation}\label{w3}
\left[ \int_{\R^{N}} \left(\int_{B^c(0,|x|/2)}g(z)dz \right)^{q}\widetilde{u}(x)dx \right]^{1/q} \leq C \int_{\R^N} g(x)\widetilde{v}(x)dx
\end{equation}
holds if and only if 
\begin{equation}\label{w4}
\tilde{A}:=\sup_{R>0}\left( \int_{B(0,R)}\widetilde{u}(x)dx \right)^{1/q}\left( \sup_{x \in B^{c}(0,R)}\widetilde{v}^{\,\,-1}(x) \right)<\infty.
\end{equation}
\end{lemma}

\begin{example}\label{example1a}
For $0 \leq \alpha<1$ and $1 \leq q<\infty$ the weights $\tilde{u}(x)=|x|^{-N-(1-\alpha)q}$ and $\tilde{v}(x)=|x|^{\alpha - 1}$   
satisfy \eqref{w2}. 
\end{example}

\begin{proof}
Let us first prove \eqref{w2} implies \eqref{w1}. Then by Minkowski inequality we have
\begin{align}
\left[ \int_{\R^{N}} \left(\int_{B(0,|x|/2)}g(z)dz \right)^{q}\widetilde{u}(x)dx \right]^{1/q} &=
%\left[ \int_{\R^{N}} \left(\int_{B(0,|x|/2)} g(z)dz \right)^{q}  \widetilde{u}(x)dx \right]^{1/q} \\
%\left(\sup_{y \in B(0,|x|/2)} \widetilde{v}^{-1}(y) \right)^{q} \widetilde{u}(x)dx \right]^{1/q} \\
  \left[ \int_{\R^{N}} \left(\int_{\R^{N}}g(z)\chi_{\left\{2|z|\leq |x|\right\}}(z,x)dz \right)^{q}  \widetilde{u}(x)dx \right]^{1/q} \nonumber \\ 
%\left(\sup_{y \in B(0,|x|/2)} \widetilde{v}^{-1}(y) \right)^{q}
 & \leq \int_{\R^{N}} \left(\int_{\R^{N}} [g(z)]^{q}\chi_{\left\{2|z|\leq |x|\right\}}(z,x) \widetilde{u}(x) dx \right)^{1/q} dz  \nonumber \\
  & =\int_{\R^{N}} g(z)\widetilde{v}(z) \left(\int_{B^{c}(0,2|z|)} \tilde{u}(x) dx \right)^{1/q} \widetilde{v}^{\,\,-1}(z)dz \nonumber \\ %\left(\sup_{y \in B(0,2|z|)} \widetilde{v}^{\,\,-1}(y) \right)dz  \\
  &\leq \tilde{A} \int_{\R^{N}} g(z)\widetilde{v}(z)dz, 
\end{align} 
since
$$ \left(\int_{B^{c}(0,2|z|)} \tilde{u}(x) dx \right)^{1/q}\widetilde{v}^{\,\,-1}(z) \leq
 \left(\int_{B^{c}(0,2|z|)} \tilde{u}(x) dx \right)^{1/q} \left(\sup_{y \in B(0,2|z|)} \widetilde{v}^{\,\,-1}(y) \right) \leq \tilde{A}. $$

Consider $\displaystyle{S(R):=\sup_{x \in B(0,R)}  \widetilde{v}^{-1}(x)}$. Then given $n \in \N$ we define
$$\widetilde{M_{n}}:= \left\{ x \in B(0,R) \,\, | \,\,  \widetilde{v}^{\,\,-1}(x)> S(R)-\frac{1}{n}\right\}. $$ 
Since $|\widetilde{M_{n}}|>0$ then there exist $M_{n}\subseteq \widetilde{M_{n}}$ with $0<|M_{n}|<\infty$. Define $\widetilde{f_{n}}(x)=\chi_{M_{n}}(x)$
thus from \eqref{w1} and $R>0$ we have
\begin{align*}
\left[ \int_{B^c(0,R)} |{M_{n}}|^{q} \widetilde{u}(x)dx \right]^{1/q} &=\left[ \int_{B^c(0,R)} \left(\int_{B(0,|x|)}\chi_{M_{n}}(z)dz \right)^{q}\widetilde{u}(x)dx \right]^{1/q} \\
& \leq \left[ \int_{\R^{N}} \left(\int_{B(0,|x|)}\chi_{M_{n}}(z)dz \right)^{q}\widetilde{u}(x)dx \right]^{1/q} \\
&\leq C \int_{M_{n}}\widetilde{v}(x)dx \\
& \leq C \left( S(R)-\frac{1}{n} \right)^{-1} |M_{n}|.
\end{align*}
Taking $n \rightarrow \infty$ we have
\begin{align*}
\left[ \int_{B^c(0,R)} \widetilde{u}(x)dx \right]^{1/q}S(R) \leq C 
\end{align*}
for all $R>0$ and then follows \eqref{w2}. 

The proof that \eqref{w3} is equivalent to \eqref{w4} follows analogously and it  will be omitted.  
\qed
\end{proof}

\begin{remark}
We emphasise that in the previous proof the weaker condition  
\begin{equation}\label{weakerw2}
 \left(\int_{B^{c}(0,2|z|)} \tilde{u}(x) dx \right)^{1/q}  \leq A\, \widetilde{v}(z) \quad \quad a.e. 
\end{equation}
is sufficient to prove \eqref{w1}.
\end{remark}

%\begin{example}\label{example1}
%Let $0<\ell<N$ and $1 \leq q<\infty$. Then $\tilde{u}(x)=|x|^{\beta q}$, $\tilde{v}(x)=|x|^{\alpha}$   
%satisfy \eqref{w2} if and only if  
%\end{example} 
%since 
%\begin{align*}
%\left( \int_{B^{c}(0,R)}\frac{u(x)}{|x|^{(N-\ell)q}}dx \right)^{1/q} &= |\mathbb{S}^{N-1}|^{1/q}\int_{R}^{\infty} r^{-\beta q-(N-\ell)q+N-1}dr
%=C R^{\alpha},
%\end{align*}
%since $-\beta q-(N-\ell+1)q+N<0$ is equivalent to $\alpha<0$ and moreover 
%$\displaystyle{\sup_{x \in B(0,R)}v^{-1}(x)=R^{-\alpha}.}$

\begin{example} Let $0<\ell<N$, $\alpha+\beta > 0$ and  
$
\displaystyle{\frac{1}{q} :=1+\frac{\beta+\alpha-\ell}{N}}
$.
Consider  $u(x)=|x|^{-\beta q}$ and $v(x)=|x|^{\alpha}$. Then \eqref{w2} holds if and only if $\alpha<0$.
\end{example} 

Indeed it follows by 
\begin{align*}
\left( \int_{B^{c}(0,R)}\frac{u(x)}{|x|^{(N-\ell)q}}dx \right)^{1/q} &= |\mathbb{S}^{N-1}|^{1/q}\int_{R}^{\infty} r^{-\beta q-(N-\ell)q+N-1}dr
=C R^{\alpha},
\end{align*}
where $C$ is independent of $R$ and $-\beta q-(N-\ell+1)q+N<0$ that is equivalent to $\alpha<0$ and moreover 
$\displaystyle{\sup_{x \in B(0,R)}v^{-1}(x)=R^{-\alpha}.}$

\bigskip

\noindent \textbf{Proof of  of Theorem \ref{thm1.1}}.
Consider $K(x,y):=\gamma_{N,\ell} |x-y|^{-N+\ell}$ and $0<\ell<N$. 
Let $\psi \in C_{c}^{\infty}(B_{1/2})$ be a cut-off function such that $0\leq \psi \leq 1$,  $\psi \equiv 1$ on $B_{1/4}$ and write $K(x,y)=K_{1}(x,y)+K_{2}(x,y)$ with
%\begin{equation}\nonumber
%\begin{array}{ll}
$\displaystyle{K_{1}(x,y)=\psi\left(\frac{y}{|x|}\right)K(x,0)}.$
%\end{array}
%\end{equation}
In order to obtain the estimate \eqref{ineq1.2} it is enough conclude that 
\begin{align*}
J_{i}\doteq &\left(\int_{\R^{N}}\left| \int_{\R^{N}} K_{i}(x,y)f(y)dy   \right|^{q} u(x) \, \,dx\right)^{1/q} 
           \lesssim \int_{\R^{N}} |f(x)|v(x)\,dx, \quad \quad i=1,2.
\end{align*}
%for every $u \in C_{c}^{\infty}(\R^{N};F)$.

Writing the definition of $K_{1}(x,y)$ and applying \eqref{w1} for $\widetilde{g}(x)=|f(x)|$, 
$\tilde{u}(x)=u(x)|x|^{-(N-\ell+1)q}$ and $\widetilde{v}(x)=v(x)$ we have
%$$J_{1}= \left(\int_{\R^{N}}\left| \int_{\R^{N}} K_{1}(x,y)f(y)dy\right|^{q}u(x)\,dx\right)^{1/q}$$ 
%and using the definition of $K_{1}(x,y)$ we have 
\begin{align*}
%J_{1}&= \left(\int_{\R^{N}}\left| \int_{\R^{N}} K_{1}(x,y)f(y)dy\right|^{q}u(x)\,dx\right)^{1/q}    \\
J_{1}&=
\left(\int_{\R^{N}}\left| \int_{\R^{N}} \psi\left(\frac{y}{|x|}\right)f(y)dy\right|^{q}{|K(x,0)|^{q}}u(x)\,dx\right)^{1/q}    \\
&=
\left(\int_{\R^{N}}\left| \int_{\R^{N}} \psi\left(\frac{y}{|x|}\right)f(y)dy\right|^{q} \frac{u(x)}{|x|^{(N-\ell)q}} dx\right)^{1/q}    \\
& \lesssim \|\psi\|_{L^\infty} \displaystyle{\left(\int_{\R^{N}}\left| \int_{B(0,{|x|/2})} f(y)dy\right|^{q}  \frac{u(x)}{|x|^{(N-\ell)q}}  \,dx\right)^{1/q}}   \\
  &\leq C  \int_{\R^{N}} f(x)v(x)dx
%&=\displaystyle{\left(\int_{\R^{N}}\left| \int_{B(0,{|x|/2})}  |y|f(y)dy\right|^{q} \frac{u(x)}{|x|^{(N-\ell+1)q}}    \,dx\right)^{1/q}}.
\end{align*}
since
\begin{align*}
\left( \int_{B^{c}(0,|x|/2)}\frac{u(y)}{|y|^{(N-\ell)q}}dy \right)^{1/q} v^{-1}(x)
\lesssim  
\left( \int_{\R^N} \frac{u(y)}{|y-x|^{(N-\ell)q}}  \,dy \right)^{1/q} v^{-1}(x) \leq \widetilde{C}, 
\end{align*}
where the constant is independent of $x$ from \eqref{hipotese} and then the weaker condition \eqref{weakerw2} is satisfied.
%\begin{equation}\label{defA}
%A:=\sup_{R>0} \left( \int_{B^{c}(0,R)}\frac{u(x)}{|x|^{(N-\ell)q}}dx \right)^{1/q}\left( \sup_{x \in B(0,R)}v^{-1}(x) \right)<\infty.
%\end{equation}
To estimate $J_{2}$ let us analyze the kernel $K_{2}(x,y)$. Clearly $K_{2}(x,y)=K(x,y)$ for $2|y|>|x|$ and  %then from  \eqref{orderk} we have 
%\begin{equation}\nonumber
%|K_{2}(x,y)|\leq C |x-y|^{\;\ell-N}, \quad 2|y|>|x|.
%\end{equation}
if $|x|>4|y|$ then $K_{2}(x,y)=K(x,y)-K(x,0)$ that implies 
\begin{align*}
|K_{2}(x,y)|\leq |y|\sup_{z \in [0,y]}|\partial_{y}K(x,z)| \leq C |y| |x|^{\;\ell-N-1}.
%&\lesssim |x-y|^{\;\ell-N}.
\end{align*}
A similar estimate holds  in the region $|x|<4|y|<2|x|$ thanks to the identity  
\begin{equation*}
K_{2}(x,y)=\left[1-\psi \left(\frac{y}{|x|}\right)\right]K(x,y)+\psi\left(\frac{y}{|x|}\right)\left[ K(x,y)-K(x,0) \right].
\end{equation*}
Using once more Minkowski's inequality we have

\begin{align*}
J_{2} = \left(\int_{\R^{N}}\left| \int_{\R^n} K_{2}(x,y)f(y)dy   \right|^{q} u(x) \, \,dx\right)^{1/q}
%= \left(\int_{\R^{N}}\left| \int_{\R^n} K_{2}(x,y)v^{-1}(y)[v(y)f(y)]dy   \right|^{q} u(x) \, \,dx\right)^{1/q}\\  
%\leq \int_{\R^{N}} \left( \int_{\R^N}\frac{|K_{2}(x,y)|^{q}}{|x|^{N-(N-\ell)q}}  \,dx \right)^{1/q} |f(y)|dy. \\
&\leq \int_{\R^{N}} \left( \int_{\R^N} \frac{u(x)}{|x-y|^{(N-\ell)q}}  \,dx \right)^{1/q} |f(y)|dy. \\
& \leq \tilde{C} \int_{\R^N} v(y)|f(y)|dy,
\end{align*}
from \eqref{hipotese}.

Now we moving on to the converse. Since $\nu \in L^{1}_{loc}(\R^N)$ let 
$y \in \R^N$ such that $\nu$ is well defined (almost everywhere set). Let  $\varphi \in C_{c}^{\infty}(B(0,1))$ nonnegative with $\int_{\R^N} \varphi(x)dx =1$ and define $\varphi_{\epsilon}(x)=\epsilon^{-N}\varphi((x-y)/\epsilon)$ for $\epsilon>0$. Applying $\varphi_{\epsilon}$ in \eqref{ineq1.2} we have
\begin{equation}\nonumber
\left( \int_{\R^{N}} | I_\ell \varphi_{\epsilon}(x)|^{q}u(x)dx\right)^{1/q} \le C \int_{\R^{N}}v(y+\epsilon z)\varphi(z)dz.
\end{equation}
Since $\displaystyle{\int_{\R^{N}}v(y+\epsilon z)\varphi(z)dz \rightarrow v(y)}$ and
$I_{\ell}\varphi_{\epsilon}(x) \rightarrow |x-y|^{-(N-\ell)}$
almost everywhere taking $\epsilon \rightarrow 0^+$ follows 
\begin{align*}
\left( \int_{\R^N} \frac{u(x)}{|x-y|^{(N-\ell)q}}  \,dx \right)^{1/q} \leq C v(y), \quad \quad \forall \,\,y \in \R^N\,\,a.e.
\end{align*}
as we desired.
\qed

\section{Stein-Weiss inequality in $L^{1}$ norm for vector fields}\label{sec3}

This section is devoted to the proof of the first part of Theorem \ref{main1}, precisely the inequality \eqref{main5}. 
The first ingredient is the following
\begin{lemma}\label{mainlemma}
Let $L(D)$ be a cocanceling operator as Definition \ref{de1.1}.   
%be an homogeneous linear differential operator of order $m$ on $\R^N$ from $F$ to $G$. 
There exists $C>0$ %and $m \in \Z^{*}_{+}$ 
such that for every $\varphi \in C^{m}(\R^{N}\backslash \left\{0\right\};F)$ that satisfies $|x|^{j}|D^{j}\varphi(x)| \in L^{1}_{loc}(\R^{N})$ for $j=0,....,m$ and for all $f \in C_{c}^{\infty}(\R^{N};F)$ such that $L(D)f=0$, we have that
\[
\left| \int_{\R^N} \varphi(y)\cdot f(y)\,dy\right|\le 
C\sum_{j=1}^m \int_{\R^N}|f(y)|\,|y|^j\, |D^j\varphi (y)|\,dy \tag{$\dagger$}.
\]
\end{lemma}
The proof follows the same steps of \cite[Lemma 2.2]{BVS} and will be present in the end of this section for a sake of completeness.  

%The previous inequality is the full version of the boundedness in $L^{1}$ norm to Hardy operator $Tg(t)=\frac{1}{t}\int_{0}^{t}g(s)ds$  

The second ingredient is a slight extension of Fundamental Lemma in \cite[Lemma 2.1]{HP}.
\begin{lemma}\label{main}
Assume $N \geq 2$, 
$0 < \ell < N$ and $K(x, y) \in L^{1}_{loc}(\R^N \times \R^N, \mathcal{L}(F;V))$ 
%a locally integrable function in $\R^N \times \R^N$ 
satisfying
\begin{equation} \label{ka} 
 |K(x, y)| \leq C \; |x - y|^{\ell-N} , \quad x \neq y 
\end{equation}
and 
\be |K(x,y)-K(x,0)| \leq C \; \frac{|y|} {|x|^{N+1-\ell}}, \label{kb} \ee
with $2|y|\leq |x|$. Suppose $0 \leq \alpha<1$, $\beta<N/q$, $\alpha+\beta>0$ and $\displaystyle{{1}/{q}=1+ ({\alpha+\beta-\ell})/{N}}$. Then if $L(D)$ is cocanceling  there exists $\widetilde{C}>0$ such that 
\begin{equation}\label{mainineq}
 \left( \int_{\R^N} \left| \int_{\R^N} K(x,y)f(y) \; dy \right|^q |x|^{-\beta q} \;dx  \right)^{1/q} 
\leq \widetilde{C} \int_{\R^{N}} |x|^{\alpha}|f(x)|dx, 
\end{equation}
for all $f \in C^\infty_c(\R^N;F)$ such that $L(D)f=0$ in the sense of distributions. %Conversely, if for all $f \in C^\infty_c(\R^N;F)$ satisfying $L(D)f=0$ the inequality \eqref{main5} holds then $L(D)$ has to be cocanceling. 
\end{lemma}
%\begin{remark}
%Clearly $K(x,y):=|x-y|^{-N+\ell}$ satisfies \eqref{ka} and \eqref{kb}. 
%\end{remark}
%\subsection{The cocanceling condition is sufficient}

The proof of inequality \eqref{main5} follows directly from Lemma \ref{main} applying $K(x,y):=\gamma_{N,\ell} |x-y|^{-N+\ell}$  at inequality \eqref{mainineq}. Clearly \eqref{ka} and \eqref{kb} are satisfied.

\begin{proof}
Let $\psi \in C_{c}^{\infty}(B_{1/2})$ be a cut-off function such that $0\leq \psi \leq 1$,  $\psi \equiv 1$ on $B_{1/4}$ and write $K(x,y)=K_{1}(x,y)+K_{2}(x,y)$ with
%\begin{equation}\nonumber
%\begin{array}{ll}
$\displaystyle{K_{1}(x,y)=\psi\left(\frac{y}{|x|}\right)K(x,0)}.$
%\end{array}
%\end{equation}
In order to obtain the estimate it is enough conclude that 
\begin{align*}
J_{i}\doteq &\left(\int_{\R^{N}}\left| \int_{\R^{N}} K_{i}(x,y)f(y)dy   \right|^{q} |x|^{-\beta q} \, \,dx\right)^{1/q} 
           \lesssim \int_{\R^{N}} |x|^{\alpha}|f(x)|\,dx, \quad \quad i=1,2,
\end{align*}
for every $u \in C_{c}^{\infty}(\R^{N};F)$. Thus,
\begin{align*}
J_{1}&=
\left(\int_{\R^{N}}\left| \int_{\R^{N}} \psi\left(\frac{y}{|x|}\right)f(y)dy\right|^{q}\frac{|K(x,0)|^{q}}{|x|^{\beta q}}\,dx\right)^{1/q}    \\
&=
\left(\int_{\R^{N}}\left| \int_{\R^{N}} \psi\left(\frac{y}{|x|}\right)f(y)dy\right|^{q}\frac{|K(x,0)|^{q}}{|x|^{N-(N-\ell)q-\alpha q }}dx\right)^{1/q}    \\
& \lesssim \displaystyle{\left(\int_{\R^{N}}\left| \int_{B_{|x|/2}} \frac{|y|}{|x|}|f(y)|dy\right|^{q}\frac{1}{|x|^{N-\alpha q}}\,dx\right)^{1/q}}   \\
&=\displaystyle{\left(\int_{\R^{N}}\left| \int_{B_{|x|/2}}  |y||f(y)|dy\right|^{q}\frac{1}{|x|^{N+(1-\alpha)q}}\,dx\right)^{1/q}}.
\end{align*}
The third inequality comes from \eqref{ka} and from a consequence of Lemma \ref{mainlemma}.
The estimate $(\dagger)$ is used for $\varphi(y)\doteq\psi(y/|x|) \eta$ where, for fixed $x$, $\eta$ is an unit vector in $F$ chosen so that %
\[
\left|\int_{\R^{N}} \psi\left(\frac{y}{|x|}\right)\eta\cdot f(y)dy \right|=
\left|\int_{\R^{N}} \psi\left(\frac{y}{|x|}\right)f(y)dy\right|.
\]
From the first part of Lemma \ref{weight1} for $\tilde{f}(y)=|y||f(y)|$, $\tilde{u}(x)=|x|^{-N-(1-\alpha)q}$ and $\tilde{v}(x)=|x|^{\alpha - 1}$ (see Example \ref{example1a}) we have   
% $\tilde{u}(x)=u(x)|x|^{-(N-\ell+1)q}$ and $\tilde{v}(x)=|x|^{-1}v(x)$ we have
\begin{align*}
J_{1} \lesssim \displaystyle{\left(\int_{\R^{N}}\left| \int_{B_{|x|/2}}  |y||f(y)|dy\right|^{q} |x|^{-N-(1-\alpha)q}  \,dx\right)^{1/q}}
&\lesssim   \int_{\R^{n}} |y| |f(y)| |y|^{\alpha - 1}dy  \\
&= \int_{\R^{n}} |f(y)||y|^{\alpha}dy.
\end{align*}

%Using Minkowski inequality and the fact that $0 \leq \alpha < 1$ we have 
%\begin{align*}
%J_{1} \lesssim \int_{\R^{N}}|y||f(y)|\left( \int_{\R^{N}\backslash B(0,2|y|)} \frac{1}{|x|^{N+q-\alpha q}}\,dx \right)^{1/q}dy 
%& \lesssim \int_{\R^{N}}|y||f(y)| {\frac{1}{|y|^{1-\alpha}}dy} \\
%& = \int_{\R^{N}}|y|^{\alpha}|f(y)|dy.
%\end{align*}

To estimate $J_{2}$ let us analyze the kernel $K_{2}(x,y)$. If $2|y|>|x|$ then from  \eqref{ka} we have $|K_{2}(x,y)|=|K(x,y)| \leq C |x-y|^{\;\ell-N}$ for $2|y|>|x|$. Otherwise if $|x|>4|y|$ then $K_{2}(x,y)=K(x,y)-K(x,0)$ that implies from \eqref{kb}
\begin{align*}
|K_{2}(x,y)|\leq |y|\sup_{z \in [0,y]}|\partial_{y}K(x,z)| \leq C |y| |x|^{\;\ell-N-1}.
%&\lesssim |x-y|^{\;\ell-N}.
\end{align*}
A similar estimate holds  in the region $|x|<4|y|<2|x|$ thanks to the identity  
\begin{equation*}
K_{2}(x,y)=\left[1-\psi \left(\frac{y}{|x|}\right)\right]K(x,y)+\psi\left(\frac{y}{|x|}\right)\left[ K(x,y)-K(x,0) \right].
\end{equation*}
Using Minkowski's inequality we have
\begin{align*}
J_{2} \leq \int_{\R^{N}} \left( \int_{\R^{N}}{|K_{2}(x,y)|^{q}}{|x|^{-\beta q}}  \,dx \right)^{1/q} |f(y)|dy %& \leq \int_{\R^{N}} \left( \int_{\R^{N}}\frac{1}{|x-y|^{(N-\ell)q}|x|^{\beta q}}  \,dx \right)^{1/q} |f(y)|dy \\
  \lesssim \int_{\R^{N}}|y|^{\alpha}|f(y)|dy,
\end{align*}
as we wished since 
\begin{align*}
\int_{\R^{N}}\frac{|K_{2}(x,y)|^{q}}{|x|^{\beta q}}\,dx & \leq 
\int_{B_{2|y|}}\frac{|x|^{\alpha q}}{|x-y|^{(N-\ell)q}|x|^{N-(N-\ell)q}}\,dx
           + \int_{\R^{N} \backslash B_{2|y|}}\frac{|y|^{q}}{|x|^{N+q-\alpha q}}\,dx
& \lesssim |y|^{\alpha q}.
\end{align*}
%so $J_{2} \lesssim\int_{\R^{N}}|y|^{\alpha}|f(y)|dy$ 
%as we wished.
\end{proof}

\begin{remark} The assumption \eqref{kb} may be replaced by the stronger assumption
\be |\partial_y K(x, y)| \leq C \; |x - y|^{\ell-N-1} , \quad x \neq y. \label{kb2} \ee
\end{remark}
\begin{corollary} Let $K_{1}(x,y)$ as in the proof of Lemma \ref{main} and $u(x),v(x)$ nonnegative functions in $L^{1}_{loc}(\R^N)$ . If
there exist an universal constant $C>0$ such that
\begin{equation}\label{pesopeso}
\left(\int_{B^{c}(0,2|y|)} \frac{u(x)}{|x|^{(N-\ell+1)q}} dx \right)^{1/q} \leq C v(y)|y|^{-1}  \quad \quad a.e. \,\, \quad y \in \R^N
\end{equation}
%$$A:=\sup_{R>0}\left( \int_{B^{c}(0,R)}\frac{u(x)}{|x|^{(N-\ell+1)q}}dx \right)^{1/q}\left( \sup_{x \in B(0,R)}v^{-1}(x)|x| \right)<\infty$$
then
$$\left(\int_{\R^{N}}\left| \int_{\R^{N}} K_{1}(x,y)f(y)dy\right|^{q}u(x)\,dx\right)^{1/q}   \lesssim  \int_{\R^N} v(y)|f(y)|dy.$$
\end{corollary}

\begin{proof}
%Writing  
%$$J_{1}= \left(\int_{\R^{N}}\left| \int_{\R^{N}} K_{1}(x,y)f(y)dy\right|^{q}u(x)\,dx\right)^{1/q}$$ 
Using the definition of $K_{1}(x,y)$ and the Lemma \ref{mainlemma}  we may write
\begin{align*}
 \left(\int_{\R^{N}}\left| \int_{\R^{N}} K_{1}(x,y)f(y)dy\right|^{q}u(x)\,dx\right)^{1/q} &
=
\left(\int_{\R^{N}}\left| \int_{\R^{N}} \psi\left(\frac{y}{|x|}\right)f(y)dy\right|^{q}{|K(x,0)|^{q}}u(x)\,dx\right)^{1/q}   \\
%&\lesssim
%\left(\int_{\R^{N}}\left| \int_{\R^{N}} \psi\left(\frac{y}{|x|}\right)f(y)dy\right|^{q} \frac{u(x)}{|x|^{(N-\ell)q}} dx\right)^{1/q}    \\
%& \lesssim \displaystyle{\left(\int_{\R^{N}}\left| \int_{B_{|x|/2}} \frac{|y|}{|x|}f(y)dy\right|^{q}  \frac{u(x)}{|x|^{(N-\ell)q}}  \,dx\right)^{1/q}}   \\
&\lesssim \displaystyle{\left(\int_{\R^{N}}\left| \int_{B_{|x|/2}}  |y||f(y)|dy\right|^{q} \frac{u(x)}{|x|^{(N-\ell+1)q}}    \,dx\right)^{1/q}}.
\end{align*}
From \eqref{w1} for $\tilde{g}(y)=|y||f(y)|$, $\tilde{u}(x)=u(x)|x|^{-(N-\ell+1)q}$ and $\tilde{v}(x)=|x|^{-1}v(x)$ assuming \eqref{pesopeso} we have
\begin{align*}
\displaystyle{\left(\int_{\R^{N}}\left| \int_{B_{|x|/2}}  |y||f(y)|dy\right|^{q} \frac{u(x)}{|x|^{(N-\ell+1)q}}    \,dx\right)^{1/q}}
\lesssim   \int_{\R^{N}} |y| |f(y)| \tilde{v}(y)dy 
= \int_{\R^{N}}  v(y)|f(y)|dy.
\end{align*}
\qed
\end{proof}

Next we present an important class of examples satisfying \eqref{pesopeso} as a limit case of the bump condition \eqref{bump-condition}.
\begin{example}
Consider $0<\ell<N$ and $q \geq 1$. Let $u \in L^{p}_{loc}(\R^{N})$ a nonnegative function for some $p \geq1$  and assume that for every ball $B:=B(x_{0},r)$ %in $\R^{N}$  
\begin{equation}\label{u3}
|B|^{ \left( \frac{1}{q}+\frac{\ell}{N}-1 \right)}\left(\frac{1}{|B|} \int_{B} u(x)^{p}dx \right)^{1/pq}\leq C
\end{equation}
where $C>0$ is independent of $u$ and $B$. Then 
$$ \left( \int_{\R^{N}\backslash B(0,2|y|)} \frac{u(x)}{|x|^{(N-\ell+1)q}}\,dx \right)^{1/q}  \lesssim |y|^{-1},\, \quad \quad a.e. \,\, \quad y \in \R^N.$$
\end{example}

Clearly if $u$ satisfies \eqref{u3} then $u \in L^{1}_{loc}(\R^N)$ and 
\begin{align*}
\left( \int_{B} u(x)dx \right)^{1/q} & \leq  \left(\int_{B} u(x)^{p}dx \right)^{1/pq} |B|^{1/q-1/pq}
%\lesssim  |B|^{-\left(\frac{1}{q}+\frac{\ell}{N}-1 \right) +\frac{1}{qr}+\frac{1}{qr'}}
\lesssim   |B|^{-\frac{\ell}{N}+1}
\end{align*}
and then
\begin{align*}
 \int_{\R^{N}\backslash B(0,2|y|)} \frac{u(x)}{|x|^{(N-\ell+1)q}}\,dx &  \leq \sum_{k=1}^{\infty} \int_{2^{k}|y|\leq |x| \leq 2^{k+1}|y|} \frac{u(x)}{|x|^{(N-\ell+1)q}}\,dx \\ 
 & \leq  \sum_{k=1}^{\infty} \frac{1}{(2^{k}|y|)^{(N-\ell+1)q}}\int_{B(0,2^{k+1}|y|)} u(x)\,dx \\
  & \lesssim  \sum_{k=1}^{\infty} \frac{1}{(2^{k}|y|)^{(N-\ell+1)q}}   (2^{k+1}|y|)^{-{\ell q}+Nq}  \\
    & \lesssim  \sum_{k=1}^{\infty} (2^{k}|y|)^{-(N-\ell+1)q-{\ell q}+Nq}  \\
     & \lesssim  |y|^{-q}.  \\
\end{align*}

\subsection{Proof of Lemma \ref{mainlemma}} 
By simplicity we identify $F$ and $V$ with standard Euclidean spaces endowed with the usual inner product that will be denoted by a dot, precisely we will write $f_1\cdot f_2$ for $f_1,f_2 \in F$. If $k \in \mathcal{L}(F,V)$ is a linear transformation from F to V, $k^* \in \mathcal{L}(V,F)$ will denote the adjoint of $k$ with respect to the inner product. We will assume throughout that $L(D)$ is cocanceling operator of order $m$ from finite dimensional complex vector space $F$ to finite dimensional complex vector space $V$.

Writing 
$\displaystyle{L(D)=\sum_{|\alpha|= m}b_{\alpha}\partial^{\alpha}}$, it follows from  \cite[Lemma 2.5]{VS} that there exist functions $k_{\alpha} \in \mathcal{L}(V,F))$ such that 
\begin{equation}\label{ident1}
\sum_{|\alpha|=m}k_{\alpha} \circ b_{\alpha}=Id_{F}.
\end{equation} 
 Let $P: \R^{N} \rightarrow \mathcal{L}(F)$ be given by 
 $\displaystyle{P(x)=\sum_{|\beta|=m}\frac{x^{\beta}}{\beta!}k_{\beta}^{\ast}}$
and let $\alpha \in \Z^{N}_{+}$ satisfy  $|\alpha|=k$.
Hence, since  $\partial^{\alpha}(x^{\beta})=0$ for $|\beta|=k$ when $\alpha \neq \beta$, we have
$\partial^{\alpha}  P(x)= k_{\alpha}^{\ast}$.
Thus, thanks to  \eqref{ident1},  the transpose $L^{*}(D)$ of $L(D)$ satisfies
 \[
L^{*}(D)(P(x))=\sum_{|\alpha|=m}b^{\ast}_{\alpha}\partial^{\alpha}(P(x))=Id_{F}.
\]
Then, given $f \in C_{c}^{\infty}(\R^{N};F)$ such that $L(D)f=0$ and $\varphi \in C_{c}^{\infty}(\R^{N};F)$ we may write
$\varphi(x)=L^{*}(D)(P(x))\varphi(x)$ and 
\begin{equation}%\label{4.13}
\int_{\R^{N}}\varphi(x)\cdot f(x)dx = \int_{\R^{N}} L^{*}(D)(P(x))\,\varphi(x) \cdot f(x)dx, \\
\end{equation}
Writing  $T_{\varphi}(x)=L^{*}(D)(P(x))\,\varphi(x) - L^{*}(D)(P(x)\varphi(x))$ we have
\[
\int_{\R^{N}}\varphi(x)\cdot f(x)dx =\int_{\R^{N}}T_{\varphi}(x)  \cdot f(x)dx
\]
since $L(D)f=0$. 
By Leibniz rule we have 
  \begin{equation}\nonumber
T_{\varphi}(x)=-  \sum_{|\alpha|=m}b_{\alpha}^{*} \sum_{0<\gamma\le\alpha} \binom{\alpha}{\gamma}\partial^{\alpha-\gamma}P(x)\partial^{\gamma}\varphi(x)
\end{equation}  
where
\[
\partial^{\alpha-\gamma}P(x)=\sum_{|\beta|=m} \frac{k^{\ast}_{\beta}}{\beta !} \partial^{\alpha-\gamma}[x^{\beta}] 
\]
Now
 \begin{equation*}
 \partial^\eta(x^{\beta})= 
 \begin{cases}
\displaystyle(\beta!/(\beta-\eta)!)x^{\beta-\eta}& \text{if $\eta \leq \beta$},\\
0&\text{otherwise}.
\end{cases}
\end{equation*}
Therefore
$$|\partial^{\alpha-\gamma}P(x)|\lesssim \sum_{|\beta|=m}\sum_{\substack{\eta \leq \alpha-\gamma\\ \eta \leq \beta}} |x|^{|\beta-\eta|} \lesssim |x|^{|\gamma|},$$
and 
$|\beta-\eta|=|\beta|-|\eta| \geq k-|\alpha-\gamma|\geq k -k +|\gamma|=|\gamma|.$
Combining the previous estimates we conclude that
\begin{equation}%\label{4.14}
|T_{\varphi}(x)|\leq C \sum_{j=1}^{m}|x|^{j}|D^{j}\varphi(x)|.
\end{equation}  
Then 
$$
\left| \int_{\R^{N}}\varphi(x)\cdot f(x)dx \right| \leq \int_{\R^{N}}|T_{\varphi}(x)| |f(x)|dx \leq 
 C \sum_{j=1}^{m} \int_{\R^{N}} |f(x)| |x|^{j}|D^{j}\varphi(x)|dx.
$$
\qed

\section{The cocancelation condition is necessary for $\alpha=0$}\label{sec4}

% \begin{theorem} \cite[Thm. 1.4]{VS}\label{teo1.4}
% Let $L(D)$ be a homogeneous linear differential operators of order k on $\R^{N}$ from $F$ to $G$. The following conditions are equivalents:
 %\begin{enumerate}
 %\item[(i)] there exists $C>0$ such that every $f \in L^{1}(\R^{n};F)$ such that $L(D)f=0$ and $\varphi \in C_{c}^{\infty}(\R^{n};F)$ the estimate
%$$ \int_{\R^N} f \cdot \varphi \leq C \|f\|_{L^{1}}|\|D\varphi\|_{L^{N}} $$ 
%holds;
 %\item[(ii)]  for every $f \in L^{1}(\R^{n};F)$ such that $L(D)f=0$ implies $\displaystyle{\int_{\mathbb{R}^{N}}f=0}$;
 %\item[(iii)] $L(D)$ is cocanceling, i.e. $\displaystyle{\bigcap_{\xi \in \R^{N}\backslash 0}ker\,L(\xi)=\left\{ 0 \right\}}$. 
 %\end{enumerate}    
 %\end{theorem}

%\textcolor{red}{I am writing the proof of this part }

The proof of the necessity for $\alpha=0$ follows the lines of \cite[Section 3]{BVS}. Consider 
\[
f \in \bigcap_{\xi \in \R^{N}\backslash \left\{ 0 \right\}}Ker \,L(\xi) \subset F.
\] 
Let $\psi \in S(\R^{N})$ satisfying $\hat{\psi}(\xi)=1$ in $B(0,1)$, so in particular  
$\int_{\R^N} \psi(x)dx =1$, and set
\[
p_{\lambda}(x):=\lambda^{N}\psi( \lambda x)-\frac{1}{\lambda^{N}}\psi\left(\frac{x}{\lambda} \right), \quad \lambda \geq 1.
\]
Clearly $\|p_{\lambda}\|_{L^{1}} \leq 2 \|\psi \|_{L^{1}}$ and 
$\widehat{p}_{\lambda}(\xi)=\widehat\psi(\xi/\lambda)-\widehat\psi(\lambda \xi)=0$ on the ball $|\xi|<1/\lambda$ for each $\lambda\ge1$.
Setting
\[
\widehat{u}_{\lambda}(\xi):=\widehat{p_{\lambda}}(\xi)f
\]
we see that
$u_{\lambda} \in S(\R^{N};F)$ because $\widehat{u}_{\lambda}(\xi)$
vanishes on a neighborhood of the origin. 
By a density argument, we may apply \eqref{main5} to $u_{\lambda}$ to get

\begin{equation}\label{main4a}
\left(\int_{\R^{N}}{| I_{\ell}u_{\lambda}(x)|^{q}}{|x|^{-\beta q}}\,dx\right)^{1/q}\, \,dx \leq {C \int_{\R^{N}}|p_{\lambda}(x)f|\,dx}
\end{equation} 
for some $C>0$. {Moreover the right hand side of previous inequality is bounded by $\tilde{C}\|\psi(x)\|_{L^1}$ with constant independent of $\lambda\ge1$.} Note that 
\begin{equation}\nonumber
I_{\ell}u_{\lambda}(x)= \left( K \ast u_{\lambda}\right)(x) = \left( K \ast p_{\lambda}\right)(x)f  
\end{equation}
where $K(x):=\gamma_{N,\ell} |x|^{-N+\ell}$, so \eqref{main4a} may be written as
\begin{equation*}
\left(\int_{\R^{N}}| K \ast p_{\lambda}(x) f|^{q}|x|^{-\beta q}\,dx\right)^{1/q}\lesssim 1. % \textcolor{blue}{\tilde{C} \lambda^{\alpha} \int_{\R^{N}}|x|^{\alpha}|\psi(x)f|\,dx.}
\end{equation*}
\noindent \textbf{Claim.}  
$\displaystyle{\lim_{\lambda \rightarrow \infty} K \ast p_{\lambda} (x)=K(x)}$ for $x \in \R^{N} \backslash \left\{ 0 \right\}$.

Therefore, assuming  Claim, it follows from Fatou's lemma that 
\begin{equation}\label{controlefinal}
\left(\int_{\R^{N}}| K(x)f|^{q}|x|^{-\beta q}\,\,dx\right)^{1/q} \lesssim 1  %\leq C \int_{\R^{N}}|p_{\lambda}(x)[f]|dx
\end{equation}
by letting $\lambda \rightarrow \infty$. However \eqref{controlefinal} holds if and only if $f=0$. Then $L(D)$ is cocanceling.

%\textcolor{red}{The same proof is bis idem in my paper with Hounie}.

To prove the Claim we may adapt the arguments in \cite[Proposition 3.1]{BVS} as we sketch below. Write $K \ast p_{\lambda}(x)-K(x)=J_{1}(\lambda,x)-J_{2}(\lambda,x)$ where
\begin{align*}
J_{1}(\lambda,x)&:=\int_{\R^{N}}\left( K(x-y)-K(x) \right)\lambda^{N}\psi(\lambda y)\,dy,\\
J_{2}(\lambda,x)&:=\int_{\R^{N}} K(x-y) \lambda^{-N}\psi(\lambda^{-1}y)\,dy.
\end{align*}
Consider first $J_{1}(\lambda,x)$.  
Taking account of the decay of $\psi$ and choosing $N<\theta<N+1$ we have 
\begin{align*}
|J_{1}(\lambda,x)| & \leq \sup_{y \in \R^{N}} \left\{ {|\lambda y|^{\theta} |\psi(\lambda y)|} \right\}\,\, \int_{\R^{N}} |K(x-y)-K(x)|\frac{\lambda^{N}}{\lambda^{\theta}|y|^{\theta}}\,dy \\
& \lesssim \lambda^{N-\theta} \int_{\R^{N}} |K(x-y)-K(x)|\frac{1}{|y|^{\theta}}\,dy.
\end{align*}
To majorize the right hand side we observe that, since  $K(x)$ is homogeneous of degree $-N+\ell$ and smooth off the origin, we have the estimates
\begin{align*}
|K(x-y)-K(x)|&\lesssim \frac{|y|}{|x|^{N-\ell+1}},\quad |y|<|x|/2,\\ 
|K(x-y)-K(x)|&\lesssim \frac{1}{|x-y|^{N-\ell}}+\frac{1}{|x|^{N-\ell}},\quad
|y|\ge |x|/2.
\end{align*}
Hence
\begin{align*}
\int_{B_{|x|/2}} |K(x-y)-K(x)|\,\frac{1}{|y|^{\theta}}\,dy &\lesssim \frac{1 } {|x|^{N-\ell+1}} \int_{B_{|x|/2}} \frac{1}{|y|^{\theta-1}}\,dy
\lesssim \frac{1} {|x|^{\theta-\ell}} 
\end{align*}
and 
\begin{align*} 
\int_{\R^{N} \backslash B_{|x|/2}} |K(x-y)-K(x)|\,\frac{1}{|y|^{\theta}}\,dy
&\lesssim \int_{\R^{N} \backslash B_{|x|/2}}\frac{1}{|x-y|^{N-\ell}}
\frac{1}{|y|^{\theta}}\,dy+\frac{1}{|x|^{N-\ell}}\int_{\R^{N} \backslash B_{|x|/2}}  \frac{1}{|y|^{\theta}}\,dy\\
& \lesssim \frac{1}{|x|^{\theta-\ell}},
\end{align*}
since $\ell <N<\theta$. Thus
\begin{align*}\nonumber
|J_{1}(\lambda,x)|  & \lesssim \lambda^{N-\theta} \frac{1}{|x|^{\theta-\ell}}.
\end{align*}
To handle $J_2$ we choose $\ell<\kappa<N$ and get
\begin{align*}
|J_{2}(\lambda,x)| \lesssim \lambda^{\kappa-N} \int_{\R^{N}} \frac{1}{|x-y|^{N-\ell}} \frac{1}{|y|^{\kappa}}  dy \leq  \lambda^{\kappa-N}  \frac{1}{|x|^{\kappa - \ell}}.
\end{align*}
We conclude that 
\[
|\left(K \ast p_{\lambda}\right)(x)-K(x)| \leq |J_{1}(\lambda,x)|+|J_{2}(\lambda,x)| \rightarrow 0,\quad x \neq 0,
\]
as $\lambda \rightarrow \infty$.
%\Qed

\section{Failure of Theorem \ref{main1} for $\alpha=1$}\label{sec5}  

Let $\varphi \in C_{c}^{\infty}(B(0,1))$ a nonnegative function with $\int_{\R^N}\varphi(x)dx\neq 1$ and $\varphi_{\epsilon}(x)=\epsilon^{-N}\varphi(x/\epsilon)$ for $\epsilon>0$. Consider the vector field $\vec{f}_{\epsilon}$ on $\R^N$ with components  
$f_{1,\epsilon}(x)=\partial_{x_{2}}(\varphi_{\epsilon}(x))$,  $f_{2,\epsilon}(x)=-\partial_{x_{1}}(\varphi_{\epsilon}(x))$ and $f_{j,\epsilon}(x)=0$ for $j=3,...,N$
that clearly satisfies $div\, \vec{f}_{\epsilon}=0$ for all $\epsilon>0$. Since divergence  operator is a cocanceling operator and  assuming the inequality \eqref{main5} holds for $\alpha=1$ we have
\begin{equation}\label{ccexample}
 \left( \int_{\R^N} \left| I_{\ell}f_{\epsilon}(x) \right|^q |x|^{-\beta q} \;dx  \right)^{1/q} 
\lesssim  \sum_{j=1,2}\int_{\R^{N}} |x||f_{j,\epsilon}(x)|dx, \quad \quad \forall \,\epsilon >0. 
\end{equation}
However $f_{j,\epsilon}(x)=(-1)^{j+1}\epsilon^{-1}(\partial_{x_{j}}\varphi)_{\epsilon}(x)$ for $j=1,2$ thus
$$  \sum_{j=1,2}\int_{\R^{N}} |x||f_{j,\epsilon}(x)|dx \lesssim  \int_{\R^{N}} |x| |\nabla \varphi (x)|dx <\infty  $$
i.e. the right side hand of \eqref{ccexample} is finite independently of $\epsilon$. 

Now we may write $I_{\ell}f_{j,\epsilon}(x)=c_{N,\ell,j}(K_{j}\ast \varphi_{\epsilon})(x)$ where $K_{j}(x)=x_{j}/|x|^{-N+\ell-2}$ and taking $\epsilon \searrow 0$ we have
\begin{equation}
 \left( \int_{\R^N} |x_{j}|^{q} |x|^{(-N+\ell-2-\beta)q} \;dx  \right)^{1/q} 
\lesssim 1
\end{equation}
 that is a contradiction for $j=1,2$ since $(-N+\ell-2-\beta)q+N=0$.

\section{Applications}\label{sec6}

\subsection{Stein-Weiss inequality in $L^{1}$ norm for canceling operators} We present a short improvement of \cite[Theorem A]{HP} concerning on a version of Hardy-Littlewood-Sobolev inequality for elliptic and canceling homogeneous linear differential operators with $0<\alpha<1$.

\begin{corollary}
Let $A(D)$ be an elliptic homogeneous linear differential operator of order $\nu$ on $\R^{N}$, $N\ge2$, from $E$ to $F$ and assume that $0\leq \alpha<1$, $0<\ell<N$ and $\ell\le\nu$. If $A(D)$ is canceling then  the estimate
\begin{equation}\label{thp}
\left(\int_{\R^{N}}| (-\Delta)^{(\nu-\ell)/2}u(x)|^{q}|x|^{-N+(N-\ell-\alpha)q}\,dx\right)^{1/q}\leq C \int_{\R^{N}}|x|^{\alpha}|A(D)u(x)|\,dx,
\end{equation} 
holds for every $u \in C_{c}^{\infty}(\R^{N};E)$,  some $C>0$ and $q \in \left[ 1, \frac{N}{N+\alpha-\ell} \right[$ .%and some $C>0$ if and only if $A(D)$ is canceling. Moreover the ellipticity is a necessary condition for $1< q<N/(N-\ell)$.
\end{corollary}

Here $g=(-\Delta)^{a/2}f$ is the positive fractional power of the Laplacian defined by the multiplier operator $\hat{g}(\xi)=|\xi|^{a}\hat{f}(\xi)$ for $f \in \mathcal{S}'(\R^N)$ and $a \geq 0$. The definition of canceling is presented at \eqref{canceling} and from \cite[Proposition 4.2]{VS} there exist a finite dimensional vector complex space V and an homogeneous differential operator $L(D)$ on $\R^N$ from $F$ to $V$ such that
\begin{equation}\label{cocan}
ker\, L(\xi)=A(\xi)[F]. 
 \end{equation}

%The proof is \textit{bis in idem} \cite[Theorem A]{HP},  that correspond the case $\alpha=0$ with use of Fundament Lemma \ref{main}. 
In order to prove \eqref{thp} first we may write $(-\Delta)^{(\nu-\ell)/2}$ as a composition product with $A(D)$ as one of the factors. Consider the function $\xi\mapsto H(\xi)\in L(F,E)$ defined by
$H(\xi)=|\xi|^{\nu-\ell}(A^{\ast}\circ A)^{-1}(\xi)A^{\ast}(\xi)$
that is smooth in $\R^{N}\backslash \left\{0\right\} $ and homogeneous of 
degree $-\ell$. %Here $A^{\ast}(\xi)$ is the symbol of adjoint operator $A^{\ast}(D)$.
Since we are assuming that $0<\ell<N$ then  $H$ is a locally integrable  tempered distribution and its inverse Fourier transform $G(x)$, i.e
$\widehat G=H$, 
is a locally integrable tempered distribution homogeneous of degree $-N+\ell$ that satisfies
\begin{equation}\label{identity1}
(-\Delta)^{(\nu-\ell)/2}u(x)=\int_{\R^N}G(x-y)[A(D)u(y)]dy, \quad \quad  u \in C_{c}^{\infty}(\R^N;E).
\end{equation}
with $\displaystyle{\bigcap_{x \in \R^N \backslash \left\{ 0 \right\}} ker\, G(x)= \bigcap_{\xi \in \R^N \backslash \left\{ 0 \right\}} ker\, A^{*}(\xi)}$.
The proof follows directly from identity \eqref{identity1} and  Lemma \ref{main} applied to $f:=A(D)u$ that belongs to the kernel of cocanceling operator $L(D)$  given at  \eqref{cocan} and  $K(x,y):=G(x-y)$ that satisfies \eqref{ka} and \eqref{kb}.

\subsection{Stein-Weiss inequality in $L^{1}$ norm for divergence operator associated to vector fields}

The following result was stated by Zhang in \cite{Z} for divergence equation of vector fields:   %second application is the Theorem 1.2 \cite{Z} due to.

\begin{theorem}\cite[Theorem 1.2]{Z}
Let $N \geq 2$, $0 < \ell < N$, $0 \leq \alpha<1$, $\beta<N/q$, $\alpha+\beta>0$ and $\displaystyle{\frac{1}{q}=1+ \frac{\alpha+\beta-\ell}{N}}$. Suppose that $K(x)$ satisfies:
\begin{enumerate}
\item[(i)] $|K(x)|\leq C |x|^{\ell-N}, \quad |x| \neq 0$;
\item[(ii)] $\displaystyle{|K(x-y)-K(x)|\leq C \frac{|y|}{|x|^{N+1-\ell}}, \quad 2|y|\leq |x|}$.
\end{enumerate}
Then if 
$$T_{\ell}f:=\int_{\R^N} K(x-y)f(y) \; dy$$
then there exists $C>0$ such that
\begin{equation}\nonumber
 \displaystyle{\| |x|^{-\beta}T_{\ell}f \|_{L^{q}} 
\leq C\| |x|^{\alpha} f\|_{L^{1}} + \| |x|^{\alpha}\nabla (-\Delta)^{-1}div\, f\|_{L^{1}}} .
\end{equation}
\end{theorem}

The result is a direct consequence of Theorem \ref{main1} taking by density $g:=f-\nabla (-\Delta)^{-1}\text{div}\, f$ 
that satisfies $L(D)g=0$ for cocanceling operator $L(D):=\text{div}$. Here
$K(x,y):=K(x-y)$ is the convolution kernel which properties \eqref{ka} and \eqref{kb}.

\subsection{Stein-Weiss inequality in Hardy spaces $H^{1}$} 
 Schikorra, Spector and Van Schaftingen proved in \cite{SSV} a very interesting inequality in $L^{1}$ norm associated to Riesz transform:

\begin{theorem}\cite[Theorem A]{SSV}
Let $N \geq 2$ and $0<\ell<N$. Then there exists a constant $C=C(\ell,N)>0$ such that
\begin{equation}\label{riesz}
 \displaystyle{ \|  I_{\ell}u \|_{L^{N/(N-\ell)}}  \leq C\| Ru\|_{L^{1}}}.
\end{equation}
for all $u \in C_{c}^{\infty}(\R^{N})$ such that $Ru \in L^{1}(\R^{N};\R^{N})$.
\end{theorem} 
\noindent Here $Ru:=(R_{1}u,...R_{N}u)$ is the vector Riesz transform where $R_{j}u(x):=(G_{j}\ast u)(x)$ is the j-Riesz transform where $\displaystyle{G_{j}(x)=c_{N}\frac{y_{j}}{|y|^{N+1}}}$ with $c_{N}:=\pi^{-\frac{N+1}{2}}\Gamma\left( \frac{N+1}{2}\right)$. 

In particular the inequality 
\eqref{riesz} recover the boundedness of Riesz potencial operator from $H^{1}(\R^N)$ to $L^{N/(N-\ell)}(\R^N)$, since the norm in  $\|f\|_{H^{1}}$ is equivalent to $\|f\|_{L^{1}}+\|Ru\|_{L^1}$.
Next we present the following extension of previous result

\begin{theorem}
Consider $0< \ell<N$, {$0\leq \alpha<1$}, $\alpha+\beta>0$ and  $\displaystyle{\frac{1}{q}=1+\frac{\alpha+\beta-\ell}{N}}$. Then 
\begin{equation}\nonumber
 \displaystyle{ \|  |x|^{-\beta} I_{\ell}u(x) \|_{L^{N/(N-\ell)}}  \leq C\| |x|^{\alpha} Ru(x)\|_{L^{1}}}
\end{equation}
for all $u \in C_{c}^{\infty}(\R^{N})$ such that $Ru \in L^{1}(\R^{N};\R^{N},|x|^{\alpha}dx)$.
\end{theorem}
\begin{proof} Follows directly from the property $\partial_{x_{i}}R_{j}u - \partial_{x_{j}}R_{i}u=0$ for $i \neq j$ i.e. the Riesz transform $Ru$ belongs to the kernel of curl operator from 1-form space $\Lambda^{1}(\R^N)$ to 2-form space $\Lambda^{2}(\R^N)$ that is a cocanceling operator.
\end{proof}

\medskip

\textbf{Acknowledgment}: The authors want to thank professor Carlos Pérez Moreno for some interesting conversations of the subject of weighted 
inequalities for fractional integrals.

\end{document}